\documentclass{amsart}
\usepackage{hyperref}
\usepackage{pstricks, pst-node,pst-tree}

\title{Front representation of set partitions}
\author{Jang Soo Kim}
\address{School of Mathematics, University of Minnesota, Minneapolis,
  Minnesota 55455, USA}
\email{kimjs@math.umn.edu}
\urladdr{\url{http://www.math.umn.edu/~kimjs}}
\thanks{The author was supported by the grant ANR08-JCJC-0011.}
\date{\today}

\newtheorem{thm}{Theorem}[section]
\newtheorem{lem}[thm]{Lemma}
\newtheorem{prop}[thm]{Proposition}
\newtheorem{cor}[thm]{Corollary}
\theoremstyle{definition}

\newtheorem{conj}{Conjecture}

\newtheorem{problem}{Problem}
\theoremstyle{remark}
\newtheorem{remark}{Remark}

\DeclareMathOperator{\fcr}{fcr}
\DeclareMathOperator{\fne}{fne}
\DeclareMathOperator{\FCR}{FCR}
\DeclareMathOperator{\FNE}{FNE}
\DeclareMathOperator{\WFCR}{WFCR}
\DeclareMathOperator{\WFNE}{WFNE}

\DeclareMathOperator{\dcr}{dcr}

\DeclareMathOperator{\dne}{dne}

\def\sch.{Schr{\"o}der}

\def\bk#1{\{#1\}}

\psset{unit=0.7cm, dotsize=5pt}
\psset{subgriddiv=1,gridcolor=lightgray,gridlabels=0pt,gridwidth=0.4pt}
\psset{radius=2pt,levelsep=.7cm,treefit=loose,treenodesize=5pt}

\def\cvput#1[#2]{\pnode(#1,1){#1} \pscircle*(#1,1){.1} \rput(#1,.5){$#2$}}
\def\vput#1{\cvput#1[#1]}
\def\edge#1#2{\ncarc[arcangle=50]{#1}{#2}}

\def\CLOSER{\pnode(0,1){vertex} \pnode(-.25,1.3){half}
\ncarc[arcangle=50]{half}{vertex} \pscircle*(0,1){.1}}
\def\closer#1{\rput(#1,0){\CLOSER}}

\newcount\ax \newcount\ay
\newcount\bx \newcount\by
\newcount\cx \newcount\cy
\newcount\dx \newcount\dy
\def\cell(#1,#2)[#3]{
\ax=#2 \ay=#1
\bx=\ax \by=\ay
\cx=\ax \cy=\ay
\dx=\ax \dy=\ay
\advance\bx by-1
\advance\dy by-1
\advance\cx by-1
\advance\cy by-1
\pspolygon(\ax,\ay)(\bx,\by)(\cx,\cy)(\dx,\dy)(\ax,\ay)
\rput(\number\cx.5,\number\cy.5){#3}
}

\begin{document}

\begin{abstract}
 Let $\pi$ be a set partition of $[n]=\{1,2,\ldots,n\}$. The standard representation of $\pi$ is the graph on the vertex set $[n]$ whose edges are the pairs $(i,j)$ of integers with $i<j$ in the same block which does not contain any integer between $i$ and $j$. The front representation of $\pi$ is the graph on the vertex set $[n]$ whose edges are the pairs $(i,j)$ of integers with $i<j$ in the same block whose smallest integer is $i$. Using the front representation, we find a recurrence relation for the number of $12\cdots k12$-avoiding partitions for $k\geq2$. Similarly, we find a recurrence relation for the number of $k$-distant noncrossing partitions for $k=2,3$.  We also prove that the front representation has several joint symmetric distributions for crossings and nestings as the standard representation does.
\end{abstract}

\maketitle

\section{Introduction}

A (set) {\em partition} of $[n]=\{1,2,\ldots,n\}$ is a collection of mutually disjoint nonempty subsets, called \emph{blocks}, of $[n]$ whose union is $[n]$.  We will write a partition as a sequence of blocks $(B_1,B_2,\ldots,B_\ell)$ such that $\min(B_1)<\min(B_2)<\cdots<\min(B_\ell)$.  Let $\Pi_n$ denote the set of all partitions of $[n]$.

Let $\pi=(B_1,B_2,\ldots,B_\ell)$ be a partition of $[n]$.  A \emph{standard edge} of $\pi$ is a pair $(i,j)$ of integers with $i<j$ such that $i$ and $j$ are in the same block, and there is no other integer $t$ with $i<t<j$ in that block. The \emph{standard representation} of $\pi$ is the graph whose vertex set is $[n]$ and edge set is the set of all standard edges of $\pi$. For example, see Fig.~\ref{fig:rep1}.  The \emph{canonical word} of $\pi$ is the word $a_1a_2\cdots a_n$, where $a_i=j$ if $i\in B_j$.  For instance, the canonical word of $(\bk{1,4,6},\bk{2,5},\bk{3},\bk{7,8,9})$ is $123121444$.

\begin{figure}
  \centering
  \begin{pspicture}(1,0.5)(9,2) \vput{1} \vput{2} \vput{3} \vput{4}
    \vput{5} \vput{6} \vput{7} \vput{8} \vput{9} \edge{1}{4}
    \edge{4}{6} \edge{2}{5} \edge{7}{8} \edge{8}{9}
  \end{pspicture}
  \caption{The standard representation of
    $(\bk{1,4,6},\bk{2,5},\bk{3},\bk{7,8,9})$}
  \label{fig:rep1}
\end{figure}
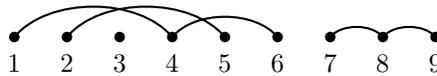

A \emph{crossing} of $\pi$ is a set of two standard edges $(i_1,j_1)$ and $(i_2,j_2)$ with $i_1<i_2<j_1<j_2$.  A \emph{noncrossing} partition is a partition without crossings. It is well known that the number of noncrossing partitions of $[n]$ is equal to the Catalan number $\frac{1}{n+1}\binom{2n}{n}$. See \cite[Exercise
  6.19]{Stanley1999} and \cite{catalan} for more examples counted by this number.

Let $\sigma=w_1w_2\cdots w_s$ be a word of integers and let $i_1,i_2,\ldots,i_t$ be the integers in $\sigma$ with $i_1<i_2<\cdots<i_t$.  The \emph{pattern} of $\sigma$ is the word obtained from $\sigma$ by replacing all the integers $i_j$ with $j$.  For instance, the pattern of $24552$ is $12331$.  For a word $\tau$, a partition $\pi\in\Pi_n$ is called \emph{$\tau$-avoiding} if the canonical word $a_1a_2\cdots a_n$ of $\pi$ does not contain a subword, i.e.~ $a_{j_1}a_{j_2}\cdots a_{j_r}$ with $j_1<j_2<\cdots<j_r$, whose pattern is $\tau$.  It is easy to see that a partition is $1212$-avoiding if and only if it is noncrossing.

For $k\geq2$, let $f_k(n)$ denote the number of $12\cdots k12$-avoiding partitions of $[n]$.  Mansour and Severini \cite{Mansour2007} found the following generating function for $f_k(n)$:
\begin{equation}\label{eq:mansour}
\sum_{n\geq0} f_k(n) x^n =
\frac{\frac{x^{k-1}y_k}{1-xy_k} + \sum_{j=0}^{k-2}\sum_{i=0}^j
  (-1)^{i+j} x^i \beta_{i,j}}
{1-\sum_{j=0}^{k-2}\sum_{i=0}^j(-1)^{i+j}ix\beta_{i,j}},  
\end{equation}
where $y_k=\frac{1-(k-2)x-\sqrt{(1-kx)^2-4x^2}} {2x(1-(k-2)x)}$, $\beta_{j,j}=1$ and $\beta_{i,j} = jx \prod_{s=i+1}^{j-1} (sx-1)$ for $i=0,1,\ldots,j-1$.

The main purpose of this paper is to find a recurrence relation for $f_k(n)$. To do this we use another representation of a partition, which is defined as follows.

Let $\pi\in\Pi_n$.  A \emph{head} of $\pi$ is the smallest element of a block of $\pi$.  A \emph{front edge} of $\pi$ is a pair $(i,j)$ of integers with $i<j$ such that $i$ and $j$ are in the same block, and $i$ is the head of that block.  The \emph{front representation}\footnote{We note that Chen et al.~\cite{Chen2008} first introduced the front representation (linear representation in their paper) to represent, so called, linked partitions.} of $\pi$ is the graph whose vertex set is $[n]$ and edge set is the set of all front edges of $\pi$.  For example, see Fig.~\ref{fig:rep2}.

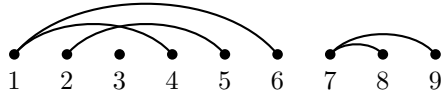
\begin{figure}
  \begin{pspicture}(1,0.5)(9,2) \vput{1} \vput{2} \vput{3} \vput{4}
    \vput{5} \vput{6} \vput{7} \vput{8} \vput{9} \edge{1}{4}
    \edge{1}{6} \edge{2}{5} \edge{7}{8} \edge{7}{9}
  \end{pspicture}
  \caption{The front representation of
    $(\bk{1,4,6},\bk{2,5},\bk{3},\bk{7,8,9})$} \label{fig:rep2}
\end{figure}

Using front edges instead of standard edges, we can define crossings as follows.

A \emph{front crossing} of $\pi$ is a set of two front edges $(i_1,j_1)$ and $(i_2,j_2)$ with $i_1<i_2<j_1<j_2$. We generalize front crossings as follows.

For $k\geq2$, a \emph{$k$-front crossing} of $\pi$ is a set of two front edges $(i_1,j_1)$ and $(i_2,j_2)$ with $i_1<i_2<j_1<j_2$ such that there are at least $k-2$ heads $h$'s with $i_2<h<j_1$.  For instance, two edges $(1,4)$ and $(2,5)$ in Fig.~\ref{fig:rep2} form a $3$-front crossing as well as a $2$-front crossing, but not a $4$-front crossing.

A \emph{$k$-front noncrossing} partition is a partition without $k$-front crossings.  It is easy to see that a partition is $k$-front noncrossing if and only if it is $12\cdots k12$-avoiding.  Thus $f_k(n)$ is the number of $k$-front noncrossing partitions of $[n]$.

Our main theorem is the following. 
\begin{thm}\label{thm:main}
Let $k\geq2$. Then $f_k(0)=f_k(1)=1$ and, for $n\geq2$, we have
\begin{align}\label{eq:main}
f_k(n) =& (k-1) f_k(n-1) + \sum_{i=1}^{k-3} (i+2-k) S(n-1,i)\\
&+\sum_{i=k-2}^{n-2} \left(f_k(i)-\sum_{j=1}^{k-3}S(i,j)\right)
\cdot \sum_{r=1}^{k-1} s(k-1,r) f_k(n-2-i+r), \notag
\end{align}
where $S(n,k)$ is the Stirling number of the second kind, i.e.~the number of partitions of $[n]$ with $k$ blocks, and $s(n,k)$ is the Stirling number of the first kind, i.e.~ $(-1)^{n-k}$ times the number of permutations of $[n]$ with $k$ cycles. 
\end{thm}

If $k=2$ then \eqref{eq:main} is equivalent to the following well known recurrence relation for the Catalan number: $f_2(0)=1$, and for $n\geq1$,
\[f_2(n) = \sum_{i=0}^{n-1} f_2(i) f_2(n-1-i).\] If $k=3$ then
\eqref{eq:main} is equivalent to the following: $f_3(0)=f_3(1)=1$, and
for $n\geq2$,
\begin{equation}\label{eq:f3}
f_3(n) = 2f_3(n-1) + \sum_{i=2}^{n-1} f_3(n-i) 
\left(f_3(i)-f_3(i-1)\right).  
\end{equation}
Using a similar argument, we can prove the following recurrence
relation for $f_2(n)$, which is very close to \eqref{eq:f3}.
\begin{thm}
  \label{thm:catalanrec}
We have $f_2(0)=f_2(1)=1$, and for $n\geq2$,
\begin{equation}\label{eq:f2}
f_2(n) = 2f_2(n-1) + \sum_{i=2}^{n-1} f_2(n-1-i) 
\left(f_2(i)-f_2(i-1)\right).
\end{equation}
\end{thm}
The argument proving \eqref{eq:f2} gives us another object counted by the Catalan number.

The rest of this paper is organized as follows. In Section~\ref{sec:hfnp} we
prove Theorems~\ref{thm:main} and \ref{thm:catalanrec}, and find a combinatorial
object counted by the Catalan number. In Section~\ref{sec:k-dist} we recall the
definition of $k$-distant noncrossing partition introduced by Drake and Kim
\cite{Drake2008}. We then find a recurrence relation for the number of
$k$-distant noncrossing partitions for $k=2,3$. We also prove that the number of
$3$-distant noncrossing partitions is equal to the sum of certain weights of
Motzkin paths, which was conjectured in \cite{Drake2008}.  In
Section~\ref{sec:sym} we review several joint symmetric distributions for
crossings and nestings in the standard representation and prove similar
properties in the front representation.

\section{A recurrence relation for $f_k(n)$}
\label{sec:hfnp}

Let us fix $k\geq2$ throughout this section. 

Let $\pi$ be a partition of $[n]$.  We say that an integer $i$ is a \emph{singleton} of $\pi$ if $\{i\}$ is a block of $\pi$. Let $A=\{a_1,a_2,\ldots,a_m\}$ be a subset of $[n]$ with $a_1<a_2<\cdots<a_m$.  We define $\pi\cap A$ to be the partition of $[m]$ obtained from $\pi$ by removing all the integers not in $A$ and replacing $a_i$ with $i$ for $i\in[m]$. We also define $\pi\setminus j$ to be $\pi\cap ([n]\setminus\{j\})$. For example, if $\pi=(\{1,2,6\},\{3,5\},\{4\})$, then $\pi\cap \{1,4,6\}=(\{1,3\},\{2\})$ and $\pi\setminus 2=(\{1,5\},\{2,4\},\{3\})$.

\subsection{A proof of the main theorem}

Let $\pi$ be a $k$-front noncrossing partition of $[n]$. Then $\pi$ falls into one of the following cases.

\begin{description}
\item[Case 1]\label{case1} $n$ is a singleton of $\pi$. Then there are $f_k(n-1)$ such $\pi$'s.
\item[Case 2] $(t,n)$ is a front edge of $\pi$ for $1\leq t<n$. Let $\pi$ have $m$ heads $h$'s with $t<h<n$. Then we have the following two sub-cases.
\begin{description}
\item[Sub-case 2-a] $0\leq m\leq k-3$. Let $\pi'=\pi\setminus n$. Then $\pi'$ is a $k$-front noncrossing partition of $[n-1]$.  Conversely, for
given $k$-front noncrossing partition $\pi'$ of $[n-1]$, we can make a partition $\pi$ in this sub-case satisfying $\pi'=\pi\setminus n$ from $\pi'$ by adding $n$ to the block containing the $j$th largest head for any $1\leq j\leq k-2$. If $\pi'$ has $i$ blocks, then there are $\min(i,k-2)$ ways to construct such a partition $\pi$. Note that if a partition has at most $k-3$ blocks, then it is a $k$-front noncrossing partition. Thus the number of $\pi$'s in this sub-case is equal to
\begin{align*}
&\sum_{i=1}^{k-3} i S(n-1,i) + (k-2)
\left( f_k(n-1) - \sum_{i=1}^{k-3} S(n-1,i)\right)\\
&=  (k-2) f_k(n-1) + \sum_{i=1}^{k-3} (i+2-k) S(n-1,i).
\end{align*}
\item[Sub-case 2-b] $m\geq k-2$. Let $\ell_0=t$, and for $1\leq r\leq k-2$, let $\ell_r$ be the $r$th smallest head among the heads greater than $t$.  Let $\pi_1=\pi\cap[\ell_{k-2}-1]$ and $\pi_2=\pi\cap\{\ell_0,\ell_1,\ldots,\ell_{k-2}, \ell_{k-2}+1,\ell_{k-2}+2,\ldots,n-1\}$.  Let $i=\ell_{k-2}-1$. Then $k-2\leq i\leq n-2$, $\pi_1\in \Pi_i$ and $\pi_2\in \Pi_{n+k-3-i}$. Note that both $\pi_1$ and $\pi_2$ are $k$-front noncrossing partitions, $\pi_1$ has at least $k-2$ blocks, and the first $k-1$ integers of $\pi_2$ are all heads.  Moreover, if $\pi_1$ and $\pi_2$ satisfy these properties, then $\pi\setminus n$ can be reconstructed from $\pi_1$ and $\pi_2$ by identifying the largest $k-2$ heads of $\pi_1$ with the smallest $k-2$ heads of $\pi_2$ and by increasing the integers greater than $k-2$ in $\pi_2$ by $i$. Thus the number of $\pi$'s in this sub-case is equal to 
\[\sum_{i=k-2}^{n-2} \left(f_k(i)-\sum_{j=1}^{k-3}S(i,j)\right)
\cdot h_k(k-1,n+k-3-i),\] where $h_k(j,n)$ is the number of
$k$-front noncrossing partitions of $[n]$ such that the first $j$ integers are
heads.
\end{description}
 \end{description}
 Summing up all the numbers in Case 1 and Case 2, we get the following
recurrence relation:
\begin{align}\label{eq:fkn1}
f_k(n) =& (k-1) f_k(n-1) + \sum_{i=1}^{k-3} (i+2-k) S(n-1,i)\\
&+\sum_{i=k-2}^{n-2} \left(f_k(i)-\sum_{j=1}^{k-3}S(i,j)\right) 
\cdot h_k(k-1,n+k-3-i).  \notag
\end{align}
Now it remains to express $h_k(m,n)$ in terms of $f_k(i)$'s.  Note that $h_k(1,n)=f_k(n)$. 

\begin{lem}
  \label{thm:5}
If $2\leq i\leq k$, then
\[h_k(i,n)= h_k(i-1,n) -(i-1)h_k(i-1,n-1).\]
\end{lem}
\begin{proof}
It can be done by a similar argument in the proof of Lemma~2.1 in \cite{Mansour2007}. More precisely, let $H_k(i,n)$ denote the set of $k$-front noncrossing partitions of $[n]$ such that the first $i$ integers are heads. By definition, we have $H_k(i,n)\subset H_k(i-1,n)$. Let $\pi\in H_k(i-1,n)\setminus H_k(i,n)$. Then $(j,i)$ is a front edge of $\pi$ for some $j\in[i-1]$. Let $\pi'=\pi\setminus i$. Then $\pi'\in H_k(i-1,n-1)$. Conversely, for each $\pi'\in H_k(i-1,n-1)$, there are $i-1$ ways to construct $\pi\in H_k(i-1,n)\setminus H_k(i,n)$. Thus we get $h_k(i,n)= h_k(i-1,n) -(i-1)h_k(i-1,n-1)$.
\end{proof}

\begin{lem}\label{thm:firstkind}
  Let $a(i,j)$ satisfy $a(i,j)= a(i-1,j) - (i-1) a(i-1,j-1)$ for
  $2\leq i\leq m$ and $1\leq j\leq n$. Then
\[a(m,n) = \sum_{r=1}^m s(m,r) a(1,n-m+r).\]
\end{lem}
To prove Lemma~\ref{thm:firstkind}, we define the weight of a lattice path from $(0,0)$ to $(n,m)$ consisting of a north step $(0,1)$ and an east step $(1,0)$ as follows. The \emph{weight} of the $i$th step is $-i$ if it is a north step; and $1$ if it is an east step. The \emph{weight} of a lattice path is the product of all of its steps.

\begin{lem}\label{thm:1}
Let $1\leq r\leq m$. Then $s(m,r)$ is the sum of the weights of all lattice
paths from $(0,0)$ to $(r-1,m-r)$.
\end{lem}
\begin{proof}
Let $b(i,j)$ be the sum of the weights of all lattice paths from $(0,0)$ to
$(j-1,i-j)$. Then $b(i,i)=1$, $b(i,1)=(-1)^{i-1}(i-1)!$ and $b(i,j)=b(i-1,j-1)-(i-1)b(i-1,j)$. Since $s(i,j)$ also satisfies these  conditions which determine $s(i,j)$ completely, we get
$b(i,j)=s(i,j)$.
\end{proof}

\begin{proof}[Proof of Lemma~\ref{thm:firstkind}]
  Using the given recurrence relation, we can express $a(m,n)$ as a linear combination of $a(m-1,n)$ and $a(m-1,n-1)$, which in turn can be expressed as a linear combination of $a(m-2,n)$, $a(m-2,n-1)$ and $a(m-2,n-2)$, and so on. Thus we can express $a(m,n)$ as 
\[a(m,n)=\sum_{r=1}^{m} b(m,r)a(1,n-(m-r)),\]
where $b(m,r)$ is the sum of the weights of all the shortest paths in Fig.~\ref{fig:lattice} from the point at which $a(m,n)$ is located to the point at which $a(1,n-(m-r))$ is located, where the weight of such a path is the product of the weights of its steps; the weight of the $i$th step is $m-i$ if it goes to the right; and $1$ if it goes to the left.  By reversing it, we can consider such a path as a lattice path from $(0,0)$ to $(r-1,m-r)$. By Lemma~\ref{thm:1}, we have $b(m,r)=s(m,r)$, which finishes the proof.
\end{proof}

\begin{figure}
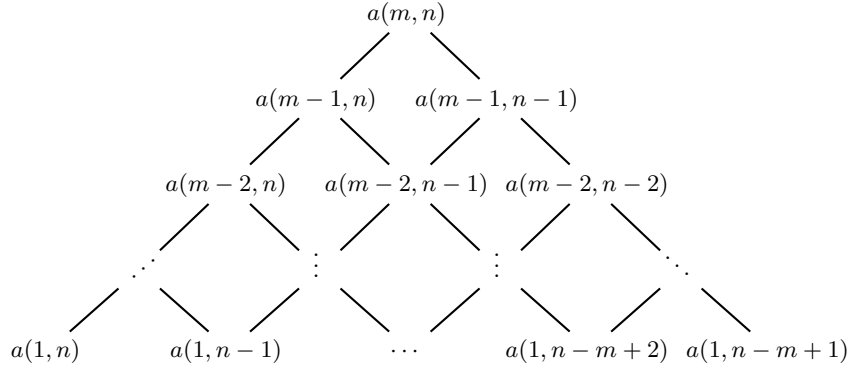

  \centering
\def\mybox#1{\makebox[.2cm]{$#1$\phantom{\vdots}}}
\[\small
\begin{psmatrix}[npos=.9,nodesep=5pt,rowsep=.5cm,colsep=.2cm,mnodesize=1cm]
&&  & & \mybox{a(m,n)} & & &&\\
&&& \mybox{a(m-1,n)} & & \mybox{a(m-1,n-1)}& &&\\
&& \mybox{a(m-2,n)} & & \mybox{a(m-2,n-1)} & & \mybox{a(m-2,n-2)} &&\\
 &\mybox{\rput{80}(0,.2){\ddots}}&&\mybox{\vdots}&&\mybox{\vdots}&&\mybox{\ddots}&\\
 \mybox{a(1,n)} && \mybox{a(1,n-1)} && \mybox{\cdots} && 
 \mybox{a(1,n-m+2)} && \mybox{a(1,n-m+1)}\\
\ncline{1,5}{2,4} \ncline{1,5}{2,6}
\ncline{2,4}{3,3} \ncline{2,4}{3,5} \ncline{2,6}{3,5} \ncline{2,6}{3,7}
\ncline{3,3}{4,2} \ncline{3,3}{4,4} \ncline{3,5}{4,4} \ncline{3,5}{4,6}
\ncline{3,7}{4,6} \ncline{3,7}{4,8}
\ncline{4,2}{5,1} \ncline{4,2}{5,3}
\ncline{4,4}{5,3} \ncline{4,4}{5,5}
\ncline{4,6}{5,5} \ncline{4,6}{5,7}
\ncline{4,8}{5,7} \ncline{4,8}{5,9}
\end{psmatrix}\]
 \caption{The recurrence relation for $a(m,n)$ gives us a diagram similar to Pascal's triangle.}
  \label{fig:lattice}
\end{figure}

By Lemma~\ref{thm:5}, Lemma~\ref{thm:firstkind} and the equation $h_k(1,j)=f_k(j)$, we get the following corollary.

\begin{cor}\label{thm:4}
If $m\leq k-1$, then 
\[h_k(m,n) = \sum_{r=1}^m s(m,r) f_k(n-m+r).\]  
\end{cor}

With Corollary~\ref{thm:4} and \eqref{eq:fkn1}, we finish the proof of Theorem~\ref{thm:main}.

As a corollary of Lemma~\ref{thm:firstkind}, we can get a result on the number of partitions with certain condition. Let $p(n)$ denote the number of partitions of $[n]$.  Let $p(m,n)$ denote the number of partitions of $[n]$ such that $1,2,\ldots,m$ are heads.

\begin{cor}
  Let $0\leq m\leq n$. Then
\[p(m,n) = \sum_{r=1}^{m} s(m,r) p(n-m+r).\]
\end{cor}
\begin{proof}
Note that $p(1,n)=p(n)$. Using the same argument in the proof of Lemma~\ref{thm:5}, we get $p(m,n)=p(m-1,n)-(m-1)p(m-1,n-1)$. By Lemma~\ref{thm:firstkind}, we are done.
\end{proof}

In fact, we can obtain the generating function for $p(m,n)$ as follows.
\begin{prop}
Let $m$ be a fixed nonnegative integer. Then
\[\sum_{n\geq0} p(m,m+n) \frac{x^n}{n!} =
\exp(mx + e^x-1).\]  
\end{prop}
\begin{proof}
It is enough to show that
\[p(m,m+n)=\sum_{i=0}^{n} \binom{n}{i} m^i p(n-i).\] Let $\pi$ be a
partition of $[m+n]$ such that $1,2,\ldots,m$ are heads.  Assume that $\pi$
has exactly $i$ front edges $(a,b)$ such that $a\leq m$ and $b\geq m+1$.
Then to construct such $\pi$, there are $m^i$ ways making these $i$ front edges and
$p(n-i)$ ways constructing a partition from the remaining integers. Thus we get the
theorem.
\end{proof}

\subsection{Another object counted by the Catalan
  number}\label{sec:tt}

Recall that $f_2(n)$ is equal to the Catalan number. The object in \textbf{6.19}
$(d^5)$ of Stanley's Catalan Addendum \cite{catalan} (version of 21 August 2010)
is equivalent to a $2$-front noncrossing partition. In this subsection we prove
Theorem~\ref{thm:catalanrec} and find a new combinatorial object counted by the
Catalan number.

\begin{proof}[Proof of Theorem~\ref{thm:catalanrec}]
Let $\pi$ be a $2$-front noncrossing partition of $[n]$ for $n\geq2$. If $n$ is a singleton of $\pi$, then there are $f_2(n-1)$ such $\pi$'s. Assume that $(j,n)$ is a front edge of $\pi$ for some $1\leq j\leq n-1$.  If there is no front edge $(j,t)$ of $\pi$ for $t<n$, then we can easily see that there are $f_2(j-1)f_2(n-1-j)$ such $\pi$'s. Otherwise let $i$ be the largest integer such that $i<n$ and $(j,i)$ is a front edge of $\pi$. Then $2\leq i\leq n-1$.  Let $\pi_1=\pi\cap[i]$ and $\pi_2=\pi\cap \{i+1,i+2,\ldots,n-1\}$.  Then both $\pi_1$ and $\pi_2$ are $2$-front noncrossing partitions, and in $\pi_1$, the largest integer $i$ is not a singleton. Conversely, $\pi$ is obtained from $\pi_1$ by adding all the blocks of $\pi_2$ whose elements are increased by $i$, and by adding $n$ to the block containing $i$. Thus there are $(f_2(i)-f_2(i-1))\cdot f_2(n-1-i)$ such $\pi$'s. Thus we get \[f_2(n)=f_2(n-1) + \sum_{j=1}^{n-1} f_2(j-1)f_2(n-1-j) + \sum_{i=2}^{n-1} f_2(n-1-i) (f_2(i)-f_2(i-1)).\] Since $f_2(n-1)=\sum_{j=1}^{n-1} f_2(j-1)f_2(n-1-j)$, we are done.
\end{proof}

A \emph{plane unary-binary tree} is an unlabeled rooted tree such that each
vertex that is not a leaf can have either a middle child, a left child, a right
child, or both left and right children.

Let $\pi$ be a $2$-front noncrossing partition of $[n]$. We define $\phi(\pi)$ as follows. If $n=1$, then $\phi(\pi)$ is the tree with only one vertex. If $n\geq2$, then $\phi(\pi)$ is recursively defined as follows. 

\begin{description}
\item[Case 1] $n$ is a singleton of $\pi$. Start with a vertex and add a middle child $v$ of it. Attach $\phi(\pi\setminus n)$ to $v$. Then $\phi(\pi)$ is the resulting tree.
\item[Case 2] $(j,n)$ is a front edge of $\pi$ and there is no front edge $(j,t)$ of $\pi$ for $t<n$. Start with a vertex and make a right child $v$ of it. Let
$\pi_1=\pi\cap[j-1]$ and $\pi_2=\pi\cap\{j+1,j+2,\ldots,n-1\}$. Let $\pi'$ be the partition obtained from $\pi_1$ by adding all the blocks of $\pi_2$ whose elements are increased by $j-1$ and by adding $n-1$ to the block containing $j$. Note that $\pi'$ is a $2$-front noncrossing partition of $[n-1]$. We attach $\phi(\pi')$ to $v$. Then $\phi(\pi)$ is the resulting tree.
\item[Case 3] $(j,n)$ is a front edge of $\pi$ and there is a front edge $(j,t)$ of $\pi$ for some $t<n$. Let $i$ be the largest integer such that $i<n$ and $(j,i)$ is a front edge of $\pi$.
  \begin{description}
  \item[Sub-Case 3-a] $2\leq i\leq n-2$. Start with a vertex and make a left child $v$ and a right child $u$ of this vertex. Attach $\phi(\pi\cap[i])$ to $v$ and $\phi(\pi\cap[i+1,n-1])$ to $u$. Then $\phi(\pi)$ is the resulting tree.
 \item[Sub-Case 3-b] $i=n-1$. Start with a vertex and make a left child $v$. Attach $\phi(\pi\cap[n-1])$ to $v$.  Then $\phi(\pi)$ is the resulting tree.
  \end{description}
\end{description}
See Fig.~\ref{fig:cat} for an example of $\phi$.

\begin{figure}
  \centering \vspace{.3cm}
  \pstree{\TC*~[tnpos=r]{
\psset{unit=.3cm,linewidth=.1pt}
\rput(-.5,-.5){\cvput{1}[] \cvput{2}[] \cvput{3}[] \cvput{4}[] \cvput{5}[] \cvput{6}[] \cvput{7}[] \cvput{8}[] \cvput{9}[] \cvput{10}[] \cvput{11}[] \edge3{11} \edge67 \edge68 \edge6{10} \edge45}
}}
{\Tn\Tn
  \pstree{\TC*~[tnpos=r]{
\psset{unit=.3cm,linewidth=.1pt}\rput(-.5,-.5){
\cvput{1}[] \cvput{2}[] \cvput{3}[] \cvput{4}[] \cvput{5}[] \cvput{6}[] \cvput{7}[] \cvput{8}[] \cvput{9}[] \cvput{10}[] \edge34 \edge3{10} \edge56 \edge57 \edge59
}}}
  {
    \pstree{\TC*~[tnpos=r]{
\psset{unit=.3cm,linewidth=.1pt}\rput(-.5,-.5){
\cvput{1}[] \cvput{2}[] \cvput{3}[] \cvput{4}[] \edge34
}}}
    {
      \Tn\Tn
      \pstree{\TC*~[tnpos=r]{
\psset{unit=.3cm,linewidth=.1pt}\rput(-.5,-.5){
\cvput{1}[] \cvput{2}[] \cvput{3}[]
}}}
      {
        \pstree{\TC*~[tnpos=r]{
\psset{unit=.3cm,linewidth=.1pt}\rput(-.5,-.5){
\cvput{1}[] \cvput{2}[]
}}}
{\TC*~[tnpos=r]{
\psset{unit=.3cm,linewidth=.1pt}\rput(-.5,-.5){
\cvput{1}[]
}}}
     }
    }
    \pstree{\TC*~[tnpos=r]{
\psset{unit=.3cm,linewidth=.1pt}\rput(-.5,-.5){
\cvput{1}[] \cvput{2}[] \cvput{3}[] \cvput{4}[] \cvput{5}[] \edge12 \edge13 \edge15
}}}
    {
      \pstree{\TC*~[tnpos=r]{
\psset{unit=.3cm,linewidth=.1pt}\rput(-.5,-.5){
\cvput{1}[] \cvput{2}[] \cvput{3}[] \edge12 \edge13
}}}
      {
        \pstree{\TC*~[tnpos=r]{
\psset{unit=.3cm,linewidth=.1pt}\rput(-.5,-.5){
\cvput{1}[] \cvput{2}[] \edge12
}}}
        {\Tn\Tn
          \TC*~[tnpos=r]{
\psset{unit=.3cm,linewidth=.1pt}\rput(-.5,-.5){
\cvput{1}[]
}}
        } \Tn
     }
      \pstree{\TC*~[tnpos=r]{
\psset{unit=.3cm,linewidth=.1pt}\rput(-.5,-.5){
\cvput{1}[] 
}}}
      {
     }
    }
  }
}
 \caption{An example of the bijection $\phi$. For each vertex $v$, the front representation of the partition corresponding  to the tree with root vertex $v$ is drawn to the right of $v$.}
  \label{fig:cat}
\end{figure}

Note that $\phi(\pi)$ is a plane unary-binary tree and $\phi(\pi)$ does not have
a left child that has a middle child or that is a leaf. Thus a left child of
$\phi(\pi)$ must have a left child or a right child (or both).

\begin{thm}\label{thm:2}
The map $\phi$ is a bijection between the set of noncrossing partitions of
  $[n]$ and the set of plane unary-binary trees with $n$ vertices such
  that a left child must have a left child or a right child (it can have both of them).
  Moreover, if $\phi(\pi)=T$ then the number of blocks of $\pi$ is equal
  to the number of leaves plus the number of vertices having a middle
  child in $T$.
\end{thm}
\begin{proof}
It is straightforward to find the inverse of $\phi$. We leave it to the reader. The `moreover' statement follows from the fact that each head corresponds to a leaf or a vertex having a middle child.
\end{proof}

\section{Similar approach to $k$-distant noncrossing partitions}
\label{sec:k-dist}

In this section we only consider the standard representation.

Let $k\geq1$.  A \emph{$k$-distant crossing} of a partition $\pi$ is a
set of two standard edges $(i_1,j_1)$ and $(i_2,j_2)$ of $\pi$ such
that $i_1<i_2<j_1<j_2$ and $j_1-i_2\geq k$. Thus a crossing is a
$1$-distant crossing.

A \emph{$k$-distant noncrossing partition} is a partition without
$k$-distant crossings. Let $d_k(n)$ denote the number of $k$-distant
noncrossing partitions of $[n]$.  Drake and Kim \cite{Drake2008} found
the following generating function for $d_2(n)$:
\begin{equation}\label{eq:d2}
\sum_{n\geq0} d_2(n) x^n = \frac{3-3x-\sqrt{1-6x+5x^2}}{2(1-x)}.  
\end{equation}
Since \eqref{eq:d2} is equal to the generating function for $f_3(n)$,
which is the special case $k=3$ of \eqref{eq:mansour}, we have
$d_2(n)=f_3(n)$. Yan \cite{Yan2008a} found bijections between
$12312$-avoiding partitions of $[n]$, UH-free \sch. paths of length
$2n+2$, and \sch. paths of length $2n+2$ without even peaks.  Kim
\cite{Kim2009b} found bijections on these objects together with
$2$-distant noncrossing partitions.

A similar argument in the proof of Theorem~\ref{thm:main} can be used to find a
recurrence relation for $d_2(n)$ and $d_3(n)$.

Recall the definitions of singleton, $\pi\cap A$ and $\pi\setminus j$ in Section~\ref{sec:hfnp}.

\subsection{The number of $2$-distant noncrossing partitions}

\begin{thm}\label{thm:2-dist}
We have $d_2(0)=d_2(1)=1$, and for $n\geq2$,
\begin{equation}
  \label{eq:recd2}
d_2(n)=2d_2(n-1) + \sum_{i=2}^{n-1} d_2(n-i) (d_2(i)-d_2(i-1)).  
\end{equation}
\end{thm}

\begin{proof}
  Let $\pi$ be a $2$-distant noncrossing partition of $[n]$ for $n\geq2$. If $n$ is a singleton of $\pi$ or $(n-1,n)$ is a standard edge of $\pi$, then there are $d_2(n-1)$ such $\pi$'s. Otherwise $(i-1,n)$ is a standard edge of $\pi$ for some $2\leq i\leq n-1$.  Let $\pi_1=\pi\cap[i]$ and $\pi_2=\pi\cap\{i,i+1,\ldots, n-1\}$. Note that $\pi$ can be reconstructed from $\pi_1$ and $\pi_2$. Both $\pi_1$ and $\pi_2$ are $2$-distant noncrossing partitions, and $(i-1,i)$ is not a standard edge of $\pi_1$. Thus the number of such $\pi_1$'s is $d_2(i)-d_2(i-1)$, and the number of such $\pi_2$'s is $d_2(n-i)$. Thus we get \eqref{eq:recd2}.
\end{proof}

Since they have the same recurrence relations \eqref{eq:f3} and \eqref{eq:recd2}, there is a recursive bijection between the set of $12312$-avoiding partitions of $[n]$ and the set of $2$-distant noncrossing partitions of $[n]$. See \cite{Kim2009b} for a direct bijection.  As a consequence of Theorem~\ref{thm:2-dist}, we can find another object counted by $d_2(n)=f_3(n)$.

A \emph{plane binary tree} is an unlabeled rooted tree such that each vertex
that is not a leaf can have a left child, a right child, or both left and right
children. A vertex of a plane binary tree is called \emph{fulfilled} if it has
both left and right children; otherwise it is called \emph{unfulfilled}.

Let $\pi$ be a $2$-distant noncrossing partition of $[n]$. We define $\psi(\pi)$ as follows.  If $n=1$, then $\psi(\pi)$ is the tree with only one vertex. If $n\geq2$, then we construct $\psi(\pi)$ recursively as follows.
\begin{description}
\item[Case 1] $n$ is a singleton of $\pi$. Start with a vertex and make a left child $v$ of it. Attach $\psi(\pi\setminus n)$ to $v$. The resulting tree is $\psi(\pi)$. 
\item[Case 2] $(n-1,n)$ is a standard edge of $\pi$.  Start with a vertex and make a right child $v$ of it. Attach $\psi(\pi\setminus n)$ to $v$. The resulting tree is $\psi(\pi)$. 
\item[Case 3] $(i-1,n)$ is a standard edge of $\pi$ for some $2\leq i\leq n-1$. Start with a vertex and make a left child $v$ and a right child $u$ of the vertex.
Attach $\psi(\pi\cap[i])$ to $v$ and $\psi(\pi\cap\{i,i+1,\ldots,n-1\})$ to $u$. 
The resulting tree is $\psi(\pi)$. 
\end{description}
For example, see Fig.~\ref{fig:2d}.

\begin{figure}
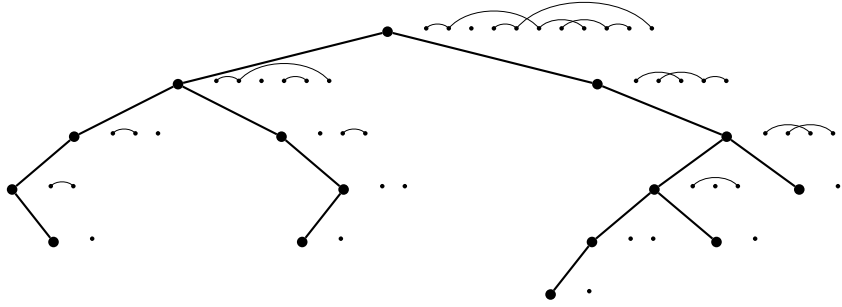

  \centering\vspace{.3cm}
  \pstree{\TC*~[tnpos=r]{
\psset{unit=.3cm,linewidth=.1pt}
\rput(-.5,-.5){\cvput{1}[] \cvput{2}[] \cvput{3}[] \cvput{4}[] \cvput{5}[] \cvput{6}[] \cvput{7}[] \cvput{8}[] \cvput{9}[] \cvput{10}[] \cvput{11}[]  \edge12 \edge26 \edge45 \edge5{11} \edge79 \edge9{10} \edge68
}}
}
{
    \pstree{\TC*~[tnpos=r]{
\psset{unit=.3cm,linewidth=.1pt}
\rput(-.5,-.5){\cvput{1}[] \cvput{2}[] \cvput{3}[] \cvput{4}[] \cvput{5}[] \cvput{6}[] 
\edge12\edge26 \edge45
}}}{
      \pstree{\TC*~[tnpos=r]{
\psset{unit=.3cm,linewidth=.1pt}
\rput(-.5,-.5){\cvput{1}[] \cvput{2}[] \cvput{3}[] \edge12 
}}}{\pstree{\TC*~[tnpos=r]{
\psset{unit=.3cm,linewidth=.1pt}
\rput(-.5,-.5){\cvput{1}[] \cvput{2}[] \edge12
}}}{\Tn\TC*~[tnpos=r]{
\psset{unit=.3cm,linewidth=.1pt}
\rput(-.5,-.5){\cvput{1}[]
}}}\Tn}
      \pstree{\TC*~[tnpos=r]{
\psset{unit=.3cm,linewidth=.1pt}
\rput(-.5,-.5){\cvput{1}[] \cvput{2}[] \cvput{3}[] \edge23
}}}{\Tn
\pstree{\TC*~[tnpos=r]{
\psset{unit=.3cm,linewidth=.1pt}
\rput(-.5,-.5){\cvput{1}[]\cvput{2}[]
}}}{\TC*~[tnpos=r]{
\psset{unit=.3cm,linewidth=.1pt}
\rput(-.5,-.5){\cvput{1}[]}}\Tn}
      }
    }
    \pstree{\TC*~[tnpos=r]{
\psset{unit=.3cm,linewidth=.1pt}
\rput(-.5,-.5){\cvput{1}[] \cvput{2}[] \cvput{3}[] \cvput{4}[] \cvput{5}[] 
\edge24\edge45 \edge13
}}}{\Tn
      \pstree{\TC*~[tnpos=r]{
\psset{unit=.3cm,linewidth=.1pt}
\rput(-.5,-.5){\cvput{1}[] \cvput{2}[] \cvput{3}[] \cvput{4}[] \edge24 \edge13
}}}{
      \pstree{\TC*~[tnpos=r]{
\psset{unit=.3cm,linewidth=.1pt}
\rput(-.5,-.5){\cvput{1}[] \cvput{2}[] \cvput{3}[] \edge13
}}}{
        \pstree{\TC*~[tnpos=r]{
\psset{unit=.3cm,linewidth=.1pt}
\rput(-.5,-.5){\cvput{1}[] \cvput{2}[]
}}}{\TC*~[tnpos=r]{
\psset{unit=.3cm,linewidth=.1pt}
\rput(-.5,-.5){\cvput{1}[]
}}\Tn}
        \TC*~[tnpos=r]{
\psset{unit=.3cm,linewidth=.1pt}
\rput(-.5,-.5){\cvput{1}[]
}}
   }
     \TC*~[tnpos=r]{
\psset{unit=.3cm,linewidth=.1pt}
\rput(-.5,-.5){\cvput{1}[] 
}}
}
    }
  }
 \caption{An example of the bijection $\psi$. For each vertex $v$, the standard representation of the partition corresponding  to the tree with root vertex $v$ is drawn to the right of $v$.}
 \label{fig:2d}
\end{figure}

Let $\pi$ is a $2$-distant noncrossing partition of $[n]$. Then $\psi(\pi)$ is a plane binary tree with $n$ unfulfilled vertices. It is not difficult to see that the left child of a fulfilled vertex must have a left child. Thus we get the following theorem.

\begin{thm}\label{thm:3}
The map $\psi$ is a bijection between the set of $2$-distant noncrossing partitions of $[n]$ and the set of plane binary trees with $n$ unfulfilled vertices satisfying the following condition: the left child of a fulfilled vertex must have a left child.
 Moreover, if $\psi(\pi)=T$, then 
 \begin{itemize}
 \item the number of blocks of $\pi$ is equal to the number of leaves plus the number of vertices with only a left child minus the number of fulfilled vertices,
 \item the number of $1$-distant crossings of $\pi$ is equal to the number of pairs $(u,v)$ of fulfilled vertices in $T$ such that $v$ is the left child of $u$.
\end{itemize}
\end{thm}
\begin{proof}
It is straightforward to find the inverse of $\psi$.  We leave it to the reader. The `moreover' statement is easy to check using induction.
\end{proof}

\subsection{The number of $3$-distant noncrossing partitions}

Let $e_n=d_3(n)$. In this subsection we will find a recurrence relation for $e_n$. 

Let $\pi$ be a $3$-distant noncrossing partition of $[n]$ for $n\geq2$. Then $\pi$ falls into one of the following cases.
\begin{description}
\item[Case 1] $n$ is a singleton of $\pi$. Then there are $e_{n-1}$ such $\pi$'s.
\item[Case 2] $(n-1,n)$ is a standard edge of $\pi$. Then there are $e_{n-1}$ such $\pi$'s.
\item[Case 3] $(n-2,n)$ is a standard edge of $\pi$. Let $\pi'=\pi\setminus n$. Then $\pi'$ is a $3$-distant noncrossing partition and $(n-2,n-1)$ is not a standard edge of $\pi'$. Thus there are $e_{n-1}-e_{n-2}$ such $\pi$'s.
\item[Case 4] $(i-2,n)$ is a standard edge of $\pi$ for some $i$ with $3\leq i\leq n-1$.  Let
$\pi_1=\pi\cap[i]$ and $\pi_2=\pi\cap \{i-1,i,\ldots,n-1\}$. Note that $\pi$ can be reconstructed from $\pi_1$ and $\pi_2$. Then $\pi_1$ and $\pi_2$ satisfy the following conditions:
\begin{itemize}
\item $\pi_1$ (resp.~$\pi_2$) is a $3$-distant noncrossing partition of $[i]$ (resp.~$[n+1-i]$).
\item Neither $(i-2,i-1)$ nor $(i-2,i)$ is a standard edge of $\pi_1$.
\item $(i-1,i)$ is a standard edge of $\pi_1$ if and only if $(1,2)$ is a standard edge of $\pi_2$.
\end{itemize}
Now we have the following two sub-cases.
\begin{description}
\item[Sub-Case 4-1] $(i-1,i)$ is a standard edge of $\pi_1$. Thus $(1,2)$ is a standard edge of $\pi_2$. Then there are $e_{i-1}-e_{i-2}$ such $\pi_1$'s
and $e_{n-i}$ such $\pi_2$'s.
\item[Sub-Case 4-2] $(i-1,i)$ is not a standard edge of $\pi_1$. Thus $(1,2)$ is not a standard edge of $\pi_2$. Then there are
$e_{n+1-i}-e_{n-i}$ such $\pi_2$'s.  We can characterize $\pi_1$ as a
$3$-distant noncrossing partition of $[i]$ such that none of $(i-2,i-1)$, $(i-2,i)$ and $(i-1,i)$ is a standard edge of $\pi_1$.
For a set $S$ of edges, let $a_E$ denote the number of
$3$-distant noncrossing partitions $\sigma$'s of $[i]$ such that all the elements of $S$ are contained in the edge set of the standard representation of $\sigma$.
Then the number of such $\pi_1$'s is equal to
$e_{i}-a_{\{(i-2,i)\}}-a_{\{(i-2,i-1)\}}-a_{\{(i-1,i)\}}+a_{\{(i-2,i-1),(i-1,i)\}}$.
Since $a_{\{(i-2,i)\}}=e_{i-1}-e_{i-2}$,
$a_{\{(i-2,i-1)\}}=a_{\{(i-1,i)\}}=e_{i-1}$ and
$a_{\{(i-2,i-1),(i-1,i)\}}=e_{i-2}$, there are $e_i-3e_{i-1}+2e_{i-2}$ such
$\pi_1$'s.
\end{description}
\end{description}

In summary, we get a recurrence relation for $e_n$. For $n\geq2$,
\begin{equation}
  \label{eq:3-dist1}
e_n = 3e_{n-1}-e_{n-2} + \sum_{i=3}^{n-1} \Big(
  e_{n-i}(e_{i-1}-e_{i-2}) + (e_{n+1-i}-e_{n-i})(e_i-3e_{i-1}+2e_{i-2}) 
\Big).  
\end{equation}

Since the summand in \eqref{eq:3-dist1} is equal to $e_{n-1}-e_{n-2}$ when $i=2$, we get the following theorem.
\begin{thm}\label{thm:3-dist}
We have that $e_0=e_1=1$, and for $n\geq2$,
\[e_n = 2e_{n-1} + \sum_{i=2}^{n-1} \Big(
e_{n-i}(e_{i-1}-e_{i-2}) + (e_{n+1-i}-e_{n-i})(e_i-3e_{i-1}+2e_{i-2})
\Big).\]
\end{thm}

Since we have a recurrence relation, we can find the generating function
from it. For a simpler calculation, let $d_{n}=e_{n+1}-e_n$,
$dd(n)=\sum_{i=0}^n d_{i}d_{n-i}$ and $ed(n)=\sum_{i=0}^n
e_{i}d_{n-i}$. Note that $d_0=dd(0)=dd(1)=ed(0)=0$,
$d_1=ed(1)=1$. Then by Theorem~\ref{thm:3-dist}, we get the following for $n\geq2$:
\begin{align*}\label{eq:dn-1}
  d_{n-1} &= e_{n-1} + \sum_{i=2}^{n-1} \Big(
  e_{n-i}d_{i-2} + d_{n-i}(d_{i-1}-2d_{i-2}) \Big)\\
  &=e_{n-1} -d_{n-2} + \sum_{i=2}^{n} \Big(
  e_{n-i}d_{i-2} + d_{n-i}(d_{i-1}-2d_{i-2}) \Big)\\
  &=e_{n-2} + \sum_{i=0}^{n-2} \Big(
  e_{i}d_{n-2-i} + d_{i}d_{n-1-i}-2d_{i}d_{n-2-i} \Big)\\
  &=e_{n-2} + ed(n-2) + (dd(n-1)-d_{n-1}d_0) -2dd(n-2)\\
  &=e_{n-2} + ed(n-2) + dd(n-1)-2dd(n-2).
\end{align*}

Let $E=\sum_{n\geq0} e_n x^n$ and $D=\sum_{n\geq0} d_n x^n$. Then
$D=\frac{E-1}x -E$ and $E=\frac{1+xD}{1-x}$. Applying the above
formula for $d_{n-1}$ to the sum $\sum_{n\geq2} d_{n-1} x^{n-1}$,
we get
\[D-d_0 = xE + xED + (D^2-dd(0)) -2xD^2.\]
Substituting $E$ with $\frac{1+xD}{1-x}$, we get
\[(1-3x+3x^2)D^2-(1-2x-x^2)D+x=0.\]
Solving this equation, we get
\[D=\frac{1-2x-x^2-(1-x)\sqrt{1-6x+x^2}}{2(1-3x+3x^2)}.\]
Finally we obtain the following generating function for $e_n$.
\begin{thm}\label{thm:gen3}
We have
\[\sum_{n\geq0} e_n x^n=\frac{2-3x+x^2-x\sqrt{1-6x+x^2}}{2(1-3x+3x^2)}.\]  
\end{thm}

Let $b=(b_0,b_1,b_2,\ldots)$ and $\lambda=(\lambda_1,\lambda_2,\ldots)$ be two infinite sequences of real numbers. The \emph{weight} of a Motzkin path with respect to $b$ and $\lambda$ is the product of $b_i$ for each horizontal step of height $i$ and $\lambda_i$ for each down step of height $i$.

Drake and Kim \cite{Drake2008} conjectured that $e_n$ is equal to the sum of weights of Motzkin paths of length $n$ with respect to $b=(b_0,b_1,b_2,\ldots)=(1,2,3,3,3\ldots)$ and $\lambda=(\lambda_1,\lambda_2,\ldots)=(1,2,2,2,\ldots)$. We can prove this
conjecture by finding the generating function for the sum of the weights of Motzkin paths. 

Let $A$ be the generating function for the sum of weights of Motzkin paths with respect to $b$ and $\lambda$. Let $A_1$ (resp.~ $A_2$) be the generating function for the sum of weights of Motzkin paths with respect to $b'=(b_1,b_2,\ldots)$ (resp.~$b''=(b_2,b_3,\ldots)$) and $\lambda'=(\lambda_2,\lambda_3,\ldots)$ (resp.~$\lambda''=(\lambda_3,\lambda_4,\ldots)$). It is easy to check that these generating functions satisfy
\begin{align*}
  A_2 &= 1+3xA_2+2x^2A_2^2,\\
  A_1 &= 1+2xA_1+2x^2A_2A_1,\\
  A &= 1+2xA+2x^2A_1A.
\end{align*}
Solving these equations, we get that $A$ is equal to the generating function in Theorem~\ref{thm:gen3}. Thus we get the following.

\begin{cor}\cite[Conjecture 6.2]{Drake2008}\label{thm:6}
  The number of $3$-distant noncrossing partitions of $[n]$ is equal
  to the sum of weights of Motzkin paths of length $n$ with respect to $b=(b_0,b_1,b_2,\ldots)=(1,2,3,3,3\ldots)$ and $\lambda=(\lambda_1,\lambda_2,\ldots)=(1,2,2,2,\ldots)$.
\end{cor}

Since both objects in Corollary~\ref{thm:6} are purely combinatorial, it would be interesting to find a combinatorial proof of Corollary~\ref{thm:6}.

\begin{problem}
Find a combinatorial proof of Corollary~\ref{thm:6}.
\end{problem}

Unfortunately, our argument is not enough to find a recurrence
relation for $d_k(n)$ for $k\geq4$. Thus we propose the following
problem.
\begin{problem}
  Find a recurrence relation for $d_k(n)$ for $k\geq4$.
\end{problem}

\section{Joint symmetric distributions for the front representation}\label{sec:sym}

In this section we recall several known results on joint symmetric
distributions of crossings and nestings in the standard
representation, and then we show that there are similar properties in
the front representation.  Since the definitions of crossings and
nestings are symmetric, even though we have already defined various
crossings, we will write them again in this section. Throughout this
section $\pi$ denotes a partition of $[n]$.

\subsection{Crossings and nestings in the standard representation}

A \emph{crossing} (resp.~\emph{nesting}) of $\pi$ is a set of two
standard edges $(i_1,j_1)$ and $(i_2,j_2)$ such that $i_1<i_2<j_1<j_2$
(resp.~$i_1<i_2<j_2<j_1$).  A \emph{noncrossing}
(resp.~\emph{nonnesting}) partition is a partition without crossings
(resp.~nestings).

It is well known that the number of noncrossing partitions of $[n]$ is
equal to the number of nonnesting partitions of $[n]$.  More
generally, Kasraoui and Zeng \cite{Kasraoui2006} proved the following:
\begin{equation}
  \label{eq:kz}
\sum_{\pi\in\Pi_n} x^{\mathrm{cr}(\pi)} y^{\mathrm{ne}(\pi)}=
\sum_{\pi\in\Pi_n} x^{\mathrm{ne}(\pi)} y^{\mathrm{cr}(\pi)},  
\end{equation}
where $\mathrm{cr}(\pi)$ (resp.~ $\mathrm{ne}(\pi)$) is the number of
crossings (resp.~ nestings) of $\pi$.

A \emph{$k$-distant crossing} (resp.~ \emph{$k$-distant nesting}) of
$\pi$ is a set of two standard edges $(i_1,j_1)$ and $(i_2,j_2)$ of
$\pi$ such that $i_1<i_2<j_1<j_2$ and $j_1-i_2\geq k$
(resp.~$i_1<i_2<j_2<j_1$ and $j_2-i_2\geq k$).

Drake and Kim \cite{Drake2008} generalized \eqref{eq:kz} as follows:
\begin{equation}
  \label{eq:dk}
\sum_{\pi\in\Pi_n} x^{\dcr_k(\pi)} y^{\dne_k(\pi)}=
\sum_{\pi\in\Pi_n} x^{\dne_k(\pi)} y^{\dcr_k(\pi)},  
\end{equation}
where $\dcr_k(\pi)$ (resp.~ $\dne_k(\pi)$) is the number of $k$-distant
crossings (resp.~ $k$-distant nestings) of $\pi$.

For $r\geq2$, Chen et al.~\cite{Chen2007} defined an
\emph{$r$-crossing} (resp.~ \emph{$r$-nesting}) of $\pi$ to be a set
of $r$ standard edges $(i_1,j_1), (i_2,j_2),\ldots,(i_r,j_r)$ of $\pi$
such that $i_1<i_2<\cdots<i_r<j_1<j_2<\cdots<j_r$ (resp.~
$i_1<i_2<\cdots<i_r<j_r<j_{r-1}<\cdots<j_1$). They showed that
\begin{equation}
  \label{eq:chen}
\sum_{\pi\in\Pi_n} x^{\mathrm{CR}(\pi)} y^{\mathrm{NE}(\pi)}=
\sum_{\pi\in\Pi_n} x^{\mathrm{NE}(\pi)} y^{\mathrm{CR}(\pi)},  
\end{equation}
where $\mathrm{CR}(\pi)$ (resp.~ $\mathrm{NE}(\pi)$) is the largest
integer $r$ such that $\pi$ has an $r$-crossing (resp.~ $r$-nesting).

\subsection{Crossings and nestings in the front representation}

For $k\geq2$, a \emph{$k$-front crossing} (resp.~\emph{$k$-front
  nesting}) of $\pi$ is a set of two front edges $(i_1,j_1)$ and
$(i_2,j_2)$ such that $i_1<i_2<j_1<j_2$ (resp.~$i_1<i_2<j_2<j_1$) and
there are at least $k-2$ heads $h$ with $i_2<h<j_1$ (resp.
$i_2<h<j_2$).

Let $\fcr_k(\pi)$ (resp.~ $\fne_k(\pi)$) denote the number of $k$-front
crossings (resp.~ $k$-front nestings) of $\pi$.  Then we have the
following analog of \eqref{eq:dk}.

\begin{thm}\label{thm:FCR}
  Let $k\geq2$. Then
\[\sum_{\pi\in\Pi_n} x^{\fcr_k(\pi)} y^{\fne_k(\pi)}=
\sum_{\pi\in\Pi_n} x^{\fne_k(\pi)} y^{\fcr_k(\pi)}.\]
\end{thm}
\begin{proof}
We will use a similar argument of Kasraoui and Zeng \cite{Kasraoui2006}, which is later generalized by Drake and Kim \cite{Drake2008}. It is also similar to the proof of Theorem~3.5 in \cite{Chen2008}.

Let $\pi\in\Pi_n$. For $i\in[n+1]$, let $T_i$ be the diagram obtained from the front representation of $\pi$ by removing integers less than $i$. For each front edge $(h,j)$ of $\pi$ with $h<i$ and $j\geq i$, we attach a \emph{half edge} to $j$. For example, see Figure~\ref{fig:half}.

\begin{figure}
  \centering
  \begin{pspicture}(1,0.5)(10,2) \vput{1} \vput{2} \vput{3} \vput{4}
    \vput{5} \vput{6} \vput{7} \vput{8}  \edge{1}{3}
    \edge16 \edge28 \edge45 \edge47 \rput(0,1){$\pi=$}
 \end{pspicture}
  \begin{pspicture}(0,0.5)(4,2) \cvput{1}[4] \cvput{2}[5] \cvput{3}[6] \cvput{4}[7]
    \cvput{5}[8] \edge12 \edge14 \closer3 \closer5 \rput(0,1){$T_4=$}
 \end{pspicture}
 \caption{$\pi$ and its $T_4$.}
  \label{fig:half}
\end{figure}
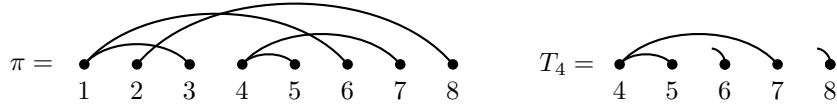

Now we define $T_i'$ for $i\in[n+1]$. Let $T_{n+1}'=\emptyset$. For $i\in[n]$, $T_i'$ is the diagram obtained from $T_{i+1}'$ by adding $i$ and by doing the following:
\begin{description}
\item[Case 1] $i$ is not a head of $\pi$. Then we attach an half edge to $i$.
\item[Case 2] $i$ is a head of $\pi$. If $i$ is a singleton of $\pi$, then doing nothing. Otherwise let $v_1,v_2,\ldots,v_r,u_1,u_2,\ldots,u_s$ be the integers which have half edges in $T_{i+1}$ such that $v_1<v_2<\cdots<v_r<u_1<u_2<\cdots<u_s$, $v_r-i<k$ and $u_1-i\geq k$. Then the construction guarantees that the integers of $T_{i+1}'$ with half edges are $v_1,v_2,\ldots,v_{r},u_1',u_2',\ldots,u_{s}'$ with $u_1'<u_2'<\cdots<u_{s}'$ and $u_1'-i\geq k$.
For each integer $t\in[r]$, if $(i,v_t)$ is an edge in $T_i$, then we make the edge $(i,v_t)$ in $T_i'$. Let $a_1,a_2,\ldots,a_p$ be the integers such that $a_1<a_2<\cdots<a_p$ and $(i,u_j)$ is an edge in $T_i$ if and only if $j\in\{a_1,a_2,\ldots,a_p\}$. Let $a_t'=a_{p+1-t}$ for $t\in[r]$. Then we make the edge $(i,u_{j}')$ in $T_i'$ for each $j\in\{a_1',a_2',\ldots,a_p'\}$.
\end{description}
We define $\pi'$ to be the partition whose front representation is $T_1'$. It is clear from the construction that the map $\pi\mapsto\pi'$ is an involution. It is not difficult to see that this involution interchanges $\fcr_k$ and $\fne_k$. 
\end{proof}

\begin{remark}\label{remark}
The case $k=2$ in Theorem~\ref{thm:FCR} is a special case of Theorem~3.5 in \cite{Chen2008}. 
\end{remark}

\begin{remark}
  One could define $f_{12\cdots k12}(\pi)$ and $f_{12\cdots k21}(\pi)$
  to be the number of subwords of the canonical word of $\pi$ which
  are order isomorphic to $12\cdots k12$ and $12\cdots k21$
  respectively. In this definition, however, the number of partitions
  $\pi$ of $[n]$ with $f_{12\cdots k12}(\pi)=s$ is not equal to that
  with $f_{12\cdots k21}(\pi)=s$. For example, let $n=5$ and $k=2$.
  If $\pi=(\bk{1,5},\bk{2,3,4})$ whose canonical word is $12221$, then
  $f_{1221}(\pi)=3$. However there is no partition $\sigma\in\Pi_5$
  with $f_{1212}(\sigma)=3$.
\end{remark}

For $r\geq2$, a \emph{front $r$-crossing} (resp.~ \emph{front
  $r$-nesting}) of $\pi$ is a set of $r$ front edges $(i_1,j_1),
(i_2,j_2),\ldots,(i_r,j_r)$ of $\pi$ such that
$i_1<i_2<\cdots<i_r<j_1<j_2<\cdots<j_r$ (resp.~
$i_1<i_2<\cdots<i_r<j_r<j_{r-1}<\cdots<j_1$).  A \emph{weak front
  $r$-crossing} (resp.~ \emph{weak front $r$-nesting}) of $\pi$ is a
set of $r$ front edges $(i_1,j_1), (i_2,j_2),\ldots,(i_r,j_r)$ of
$\pi$ such that $i_1\leq i_2\leq\cdots\leq i_r<j_1<j_2<\cdots<j_r$
(resp.~ $i_1\leq i_2\leq\cdots\leq i_r<j_r<j_{r-1}<\cdots<j_1$).

Let $\FCR(\pi)$ (resp.~ $\FNE(\pi)$) denote the largest integer $r$
such that $\pi$ has a front $r$-crossing (resp.~ front $r$-nesting).
Let $\WFCR(\pi)$ (resp.~ $\WFNE(\pi)$) denote the largest integer $r$
such that $\pi$ has a weak front $r$-crossing (resp.~ weak front
$r$-nesting).  

Now we will recall Krattenthaler's results on fillings of Ferrers diagrams \cite{Krattenthaler2006}.  Let us first define some notions in \cite{Krattenthaler2006}. For a weakly decreasing sequence of positive integers $\lambda=(\lambda_1,\ldots,\lambda_\ell)$, the \emph{Ferrers diagram} of $\lambda$ is the left-justified array of square cells such that the $i$th row has $\lambda_i$ cells.  We will draw a Ferrers diagram in French notation, i.e.~ the first row is placed at the bottom.

A \emph{0-1 filling} of a Ferrers diagram is an assignment of $0$ or $1$ to each cell.  An \emph{ne-chain} (resp.~\emph{se-chain}) of a 0-1 filling is a sequence $c_1,\ldots,c_\ell$ of cells containing $1$, such that $c_{i+1}$ is strictly north-east (resp.~south-east) of $c_i$ for each $i\in[\ell-1]$, and the smallest rectangle containing $c_1$ and $c_\ell$ is fully contained in $F$. In the same manner, we define nE-, sE-, Ne-, Se-, NE- and SE-chains, where `n',`s' and `e' mean strictly north, south and east and `N',`S' and `E' mean weakly north, south and east.

Let $F$ be a Ferrers diagram with $p$ columns and $q$ rows.  Let
$\mathbf{c}=(c_1,\ldots,c_p)$ and $\mathbf{r}=(r_1,\ldots,r_q)$ be
sequences of nonnegative integers and let $s$ and $t$ be nonnegative
integers.  Let
$N^{01}(F,\mathbf{c},\mathbf{r};\mathrm{ne}=s,\mathrm{se}=t)$ denote the
number of 0-1 fillings of $F$ such that the $i$th column has $c_i$
$1$'s, the $j$th row has $r_j$ $1$'s, and the maximum lengths of an ne-chain and an se-chain are $s$ and $t$ respectively. We define
$N^{01}(F,\mathbf{c},\mathbf{r};\mathrm{nE}=s,\mathrm{Se}=t)$ and
$N^{01}(F,\mathbf{c},\mathbf{r};\mathrm{Ne}=s,\mathrm{sE}=t)$ in the same
way.

Krattenthaler \cite[Theorem 13]{Krattenthaler2006} proved the following:
\begin{equation}
  \label{eq:krat}
N^{01}(F,\mathbf{c},\mathbf{r};\mathrm{nE}=s,\mathrm{Se}=t)
=N^{01}(F,\mathbf{c},\mathbf{r};\mathrm{Ne}=t,\mathrm{sE}=s).
\end{equation}
We note that \eqref{eq:krat} is slightly stronger
than the original one. However,
it is easy to see that Krattenthaler's argument proves \eqref{eq:krat}.

We can identify $\pi\in\Pi_n$ with a 0-1 filling of the Ferrers diagram $(n-1,n-2,\ldots,1)$ as follows.  For each front edge $(i,j)$ of $\pi$, we fill the cell in the $i$th column and in the $(n-j)$th row with $1$. We fill the remaining cells with $0$.  For example, see Fig.~\ref{fig:ferres2}.

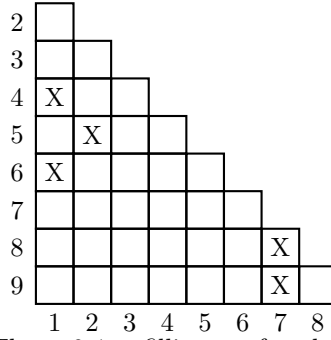
\begin{figure}
  \centering
\psset{unit=0.5cm}
\begin{pspicture}(0,9) (9,0) \cell(1,1)[] \cell(1,2)[] \cell(1,3)[]
  \cell(1,4)[] \cell(1,5)[] \cell(1,6)[] \cell(1,7)[X] \cell(1,8)[]
  \cell(2,1)[] \cell(2,2)[] \cell(2,3)[] \cell(2,4)[] \cell(2,5)[]
  \cell(2,6)[] \cell(2,7)[X] \cell(3,1)[] \cell(3,2)[] \cell(3,3)[]
  \cell(3,4)[] \cell(3,5)[] \cell(3,6)[] \cell(4,1)[X] \cell(4,2)[]
  \cell(4,3)[] \cell(4,4)[] \cell(4,5)[] \cell(5,1)[] \cell(5,2)[X]
  \cell(5,3)[] \cell(5,4)[] \cell(6,1)[X] \cell(6,2)[] \cell(6,3)[]
  \cell(7,1)[] \cell(7,2)[] \cell(8,1)[] 
 \rput(0.5,-0.5){$1$} \rput(1.5,-0.5){$2$}
  \rput(2.5,-0.5){$3$} \rput(3.5,-0.5){$4$} \rput(4.5,-0.5){$5$}
  \rput(5.5,-0.5){$6$} \rput(6.5,-0.5){$7$} \rput(7.5,-0.5){$8$}
 \rput(-0.5,0.5){$9$} \rput(-0.5,1.5){$8$}
  \rput(-0.5,2.5){$7$} \rput(-0.5,3.5){$6$} \rput(-0.5,4.5){$5$}
  \rput(-0.5,5.5){$4$} \rput(-0.5,6.5){$3$} \rput(-0.5,7.5){$2$}
\end{pspicture}
\caption{The 0-1 filling of the Ferrers diagram $(8,7,\ldots,1)$ corresponding to the front
  representation of $(\bk{1,4,6},\bk{2,5},\bk{3},\bk{7,8,9})$. For better visibility, we write $X$'s instead of $1$'s and omit $0$'s.}
  \label{fig:ferres2}
\end{figure}

It is easy to check that a 0-1 filling can be obtained in this way if
and only if each row has at most one 1, and if the $i$th column has a
1, then the $(n+1-i)$th row does not have a 1.
In other words, if $c_i$ and $r_i$ are the number of 1's in
the $i$th column and the $i$th row respectively, then for all $i$,
\begin{equation}\label{eq:condition}
r_i\in\{0,1\},\qquad c_{i}\cdot r_{n+1-i}=0.
\end{equation}

In this identification, a front $r$-crossing (resp.~front $r$-nesting)
corresponds to an se-chain (resp.~ne-chain) and a weak front
$r$-crossing (resp.~weak front $r$-nesting) corresponds to an Se-chain
(resp.~Ne-chain). Observe that if each row has at most one $1$, then ne- (resp.~ Ne-,
se-, Se-) chains are equivalent to nE- (resp.~ NE-, sE-, SE-) chains.

Thus the number of $\pi\in\Pi_n$ with $\FCR(\pi)=s$ and $\WFNE(\pi)=t$
is equal to
\begin{align*}
\sum_{\mathbf{c},\mathbf{r}} N^{01}(F,\mathbf{c},\mathbf{r};\mathrm{ne}=s,\mathrm{Se}=t)
&= \sum_{\mathbf{c},\mathbf{r}} N^{01}(F,\mathbf{c},\mathbf{r};\mathrm{nE}=s,\mathrm{Se}=t)\\
&= \sum_{\mathbf{c},\mathbf{r}} N^{01}(F,\mathbf{c},\mathbf{r};\mathrm{Ne}=t,\mathrm{sE}=s)\\
&= \sum_{\mathbf{c},\mathbf{r}} N^{01}(F,\mathbf{c},\mathbf{r};\mathrm{Ne}=t,\mathrm{se}=s),
\end{align*}
where $F$ is the Ferrers diagram of $(n-1,n-2,\ldots,1)$, and the sum is over
all the sequences $\mathbf{c}=(c_1,c_2,\ldots,c_{n-1})$ and
$\mathbf{r}=(r_1,r_2,\ldots,r_{n-1})$
satisfying \eqref{eq:condition}. Since this number is equal to the
number of $\pi\in\Pi_n$ with $\WFCR(\pi)=t$ and $\FNE(\pi)=s$, we get
the following analog of \eqref{eq:chen}.

\begin{thm}\label{thm:symmetric}
We have
\[\sum_{\pi\in\Pi_n} x^{\FCR(\pi)} y^{\WFNE(\pi)}=
\sum_{\pi\in\Pi_n} x^{\FNE(\pi)} y^{\WFCR(\pi)}.\]
\end{thm}

\begin{cor}\label{cor:sym}
  Let $r\geq2$. 
  \begin{enumerate}
  \item The number of partitions of $[n]$
    without front $r$-crossings is equal to the number of partitions of
  $[n]$ without front $r$-nestings.
  \item The number of partitions of $[n]$ without weak front
    $r$-crossings is equal to the number of partitions of $[n]$ without
    weak front $r$-nestings.
  \end{enumerate}
\end{cor}
The first item in Corollary~\ref{cor:sym} is equivalent to this: the
number of $12\cdots r 12\cdots r$-avoiding partitions of $[n]$ is
equal to the number of $12\cdots r r\cdots 21$-avoiding
partitions of $[n]$, which is a special case of Corollary~20
in \cite{Jelinek2008}.

We propose the following conjecture, which has been checked up to $n=11$.
\begin{conj}\label{conj:sym}
We have
\[\sum_{\pi\in\Pi_n} x^{\FCR(\pi)} y^{\FNE(\pi)}=
\sum_{\pi\in\Pi_n} x^{\FNE(\pi)} y^{\FCR(\pi)}.\]
\end{conj}

However, it is not the case when we use $\WFCR$ and $\WFNE$.
We have the following:
\begin{align*}
\#\{\pi\in \Pi_8 : \WFCR(\pi)=4, \WFNE(\pi)=2\} &= 73,\\
\#\{\pi\in \Pi_8 : \WFNE(\pi)=2, \WFCR(\pi)=4\} &= 74.
\end{align*}

\section*{Acknowledgement}
The author would like to thank Anisse Kasraoui for the information of Chen et al.'s work \cite{Chen2008}.


\begin{thebibliography}{10}

\bibitem{Chen2007}
William Y.~C. Chen, Eva Y.~P. Deng, Rosena R.~X. Du, Richard~P. Stanley, and
  Catherine~H. Yan.
\newblock Crossings and nestings of matchings and partitions.
\newblock {\em Trans. Amer. Math. Soc.}, 359(4):1555--1575 (electronic), 2007.

\bibitem{Chen2008}
William~Y.C. Chen, Susan~Y.J. Wu, and Catherine~H. Yan.
\newblock Linked partitions and linked cycles.
\newblock {\em European J. Combin.}, 29:1408--1426, 2008.

\bibitem{Drake2008}
Dan Drake and Jang~Soo Kim.
\newblock $k$-distant crossings and nestings of partitions.
\newblock \url{http://arxiv.org/abs/0812.2725}.

\bibitem{Jelinek2008}
V{\'i}t Jel{\'i}nek and Toufik Mansour.
\newblock On pattern-avoiding partitions.
\newblock {\em Electron. J. Combin.}, 15(1):Research paper 39, 52, 2008.

\bibitem{Kasraoui2006}
Anisse Kasraoui and Jiang Zeng.
\newblock Distribution of crossings, nestings and alignments of two edges in
  matchings and partitions.
\newblock {\em Electron. J. Combin.}, 13(1):Research Paper 33, 12 pp.
  (electronic), 2006.

\bibitem{Kim2009b}
Jang~Soo Kim.
\newblock Bijections on two variations of noncrossing partitions.
\newblock \url{http://arxiv.org/abs/0812.4091}.

\bibitem{Krattenthaler2006}
C.~Krattenthaler.
\newblock Growth diagrams, and increasing and decreasing chains in fillings of
  {F}errers shapes.
\newblock {\em Adv. in Appl. Math.}, 37(3):404--431, 2006.

\bibitem{Mansour2007}
T.~Mansour and S.~Severini.
\newblock Enumeration of $(k,2)$-noncrossing partitions.
\newblock {\em Discrete Math.}, 300(20):4570--4577, 2008.

\bibitem{catalan}
Richard Stanley.
\newblock Catalan addendum.
\newblock \url{http://math.mit.edu/~rstan/ec/catadd.pdf}.

\bibitem{Stanley1999}
Richard~P. Stanley.
\newblock {\em Enumerative Combinatorics. {V}ol. 2}, volume~62 of {\em
  Cambridge Studies in Advanced Mathematics}.
\newblock Cambridge University Press, Cambridge, 1999.
\newblock With a foreword by Gian-Carlo Rota and appendix 1 by Sergey Fomin.

\bibitem{Yan2008a}
Sherry~H.F. Yan.
\newblock Schr{\"o}der paths and pattern avoiding partitions.
\newblock \url{http://arxiv.org/abs/0805.2465v2}.

\end{thebibliography}
\end{document}